\documentclass[11pt]{amsart}
\usepackage{amsmath,amssymb, latexsym, amsthm ,graphicx,cite}

\newtheorem{theorem}{Theorem}[section]
\newtheorem*{theorem*}{Theorem}
\newtheorem{proposition}[theorem]{Proposition}
\newtheorem{algorithm}[theorem]{Algorithm}
\newtheorem*{proposition*}{Proposition}
\newtheorem{lemma}[theorem]{Lemma}

\newtheorem*{corollary*}{Corollary}

\newtheorem{remark}[theorem]{Remark}

\theoremstyle{definition}
\newtheorem{example}[theorem]{Example}

\newtheorem{problem}[theorem]{Problem} 

\graphicspath{ {/} }

\author{Avichai Tendler}
\author{Uri Alon}

\email{tendlea@gmail.com, urialonw@gmail.com}

\thanks{}

\def\RR{{\mathbb R}}

\def\PP{{\mathbb P}}

\begin{document}

\title{Approximating Functions on Boxes}

\begin{abstract}
The vector space of all polynomial functions of degree $k$ on a box of dimension $n$ is of dimension ${n \choose k}$. A consequence of this fact is that a function can be approximated on vertices of the box using other vertices to higher degrees than expected. This approximation is useful for various biological applications such as predicting the effect of a treatment with drug combinations and computing values of fitness landscape.   
\end{abstract}

\date{\today}

\maketitle
\section{Introduction}

The process of drug discovery is challenging and expensive \cite{Morgan2011}, but even while existing drugs might not bring a cure, sometimes a combination of two or more drugs might act synergistically and work better than expected by the individual effects \cite{DeVita1975}. Assume we have $n$ different drugs (e.g. antibiotics) and we want to use an effective drug combination. Usually it is infeasible to measure the effect of all $2^n$ possible combinations, hence it is useful to measure only a subset of this exponential space and predict the rest, for example we can measure the effect of only $n$ singles and ${n \choose 2}$ pairs and try to extrapolate \cite{Wood2012}. Another related relevant question is which subset of the space to measure in order to get an optimal approximation for the entire space.

Another example is an estimation of fitness landscapes \cite{Jin2005}. Assume we want to estimate the dependency of a fitness of an organism on its genome, if there are $n$ possible different mutations, there will be $2^n$ possible genomes. We wish to approximately map the entire fitness landscape without making all  mutations explicitly in the lab. Which mutation we should have in order to obtain a good approximation of the entire fitness landscape? We will also treat a common experimental situation, where we can only get random mutations, how many mutations will be needed to get a given approximation of the entire fitness landscape?  

Both the drug combination and fitness landscape problems (and others), boil down into an approximation of functions on box vertices. The different drug combinations effects or fitness landscape values are values of a function on vertices of a box. We are given values of this function on some of vertices of the box, and we wish to estimate it on the other vertices. Another problem is choosing a set of vertices which well approximate the rest. In this paper we treat these problems from algebro-geometric perspective. Interestingly, because of the fact that all polynomial functions on hypercube are spanned by the set of square-free monomials, function estimations using values on box vertices are "better than expected".

Here we compute the minimal number of values of a function on box vertices necessary in order to obtain estimations of the function on all vertices of the box. We also give a linear-algebra-based algorithm to test whether  a given set of vertices are enough to estimate a function on all vertices to a given order. Besides, we compute and simulate probabilities of random sets of vertices to estimate a function to the first order, and we show that in general, a random set of points is good for estimation with high probability. We formalize these statements below.

\section{Notation and Problem formulation}
We work over the field $\RR$ since this is the relevant field for most applications. Some of the results are valid for other fields.

We are interesting in the question of approximating a suitably differentiable function $f:\RR^n\to\RR$ on a point $t\in \RR^n$ using its values on other points $S\subseteq\RR^n$. To be more precise, to which order in "Taylor series" a function can be approximated at $t$ assuming only its values on the set $S$ are known.

The Taylor polynomial of degree $k$ of a function $f$ at a point $p$, is a polynomial $g$ of degree $k$ with $g^{(i)}(p)=f^{(i)}(p)$ for $i=0\dots k$. This is the unique polynomial of degree up to $k$ which satisfies the above equalities. Similarly, we are asking if the values of the function $f$ on the set $S$ determine a unique polynomial $g$ of degree $k$ such that $g(p)=f(p)$ for all $p\in S$.

Even if the polynomial $g$ is not unique, its value at a given specific point $t$ can be sometimes determined uniquely. This motivates the following algebro-geometric formalization of Taylor approximation:

\begin{problem}
Given a set of points $S\subseteq\RR^n$ and a point $t$, find the maximal $k$ such that any polynomial of degree up to $k$ which vanishes on $S$, vanishes also on $t$. 
\end{problem}

From now on, this is what we will mean when we say "A function $f$ can be approximated to the $k$-th order at $t$ using its values on $S$". We denote it by $k=deg(S\to t)$.

Note that this definition is equivalent to the statement that the values of a polynomial of degree $k$ on $S$, determine its value on $t$. 

\section{Preliminaries: Points in non-general position and the Cayley-Bacharach theorem}

If the set of points $S\cup \{t\}$ is in general position, the maximal degree of approximation $k=deg(S\to t)$ can be computed using counting arguments. There are $\binom{n+k}{k}$ polynomials of degree up to $k$ (same as homogenous forms of degree $k$ in $\PP^n$), therefore a polynomial of degree $k$ is determined by $\binom{n+k}{k}$ points. Thus, in a single variable $n=1$ a line is determined by 2 points, a quadratic by 3 points etc. while for $n=2$ a plane is determined by 3 points, and a conic by 6 points. Conversely, values of plane quadratic polynomial on 6 points in general position, determine the polynomial uniquely.

Interestingly, there are degenerate cases in which fewer points are enough to obtain the same degree of approximation. An example is the Cayley-Bacharach theorem \cite{Bacharach1886, Eisenbud1996}: 
\begin{theorem}
Let $f_1,f_2\subseteq\PP^2$ be two cubic plane curves meeting at nine points $s_1,...s_8,t$. If $f\subseteq\PP^2$ is any cubic containing $s_1,...s_8$, then $f$ contains also $t$.
\end{theorem}

The Cayley-Bacharach theorem implies that if we take $S=\{s_1,\dots,s_8\}$ of the theorem, then $deg(S\to t)=3$. This is nontrivial since $\|S\|=8$ and for points in general position we will usually need $\|S\|=\binom{2+3}{3}=10$ to get  $deg(S\to t)=3$.

\section{A function on box vertices can be estimated to "higher order than expected"}
The following theorem is of the Cayley-bacharach theorem, it states that the vertices of the hypercube do surprisingly well in approximating one another. 

\begin{theorem}\label{mainThm}
Let $V$ be the set of vertices of an $n$ dimensional box. Let $k<n$, then there exists a (non-unique) subset $U\subseteq V$ such that $\|U\|=\sum_{i=0}^k{n \choose i}$ and $k=deg(U\to v)$ for any $v\in V-U$.
\end{theorem}

This theorem is equivalent to the fact that the set of square-free monomials form a basis for the polynomials on the hypercube. Although this fact is known, we provide a proof here for completeness.

This result is non-trivial, since to evaluate a polynomial of degree $k$ on a general point  in $n$-dimensional space we need $\binom{n+k}{k}$ points. The theorem states that box vertices are special, hence $\sum_{i=0}^k{n \choose i}$  are enough. This reduction in the number of points is illustrated in Figure \ref{Theorem_example}.

\begin{figure}
\centering
\includegraphics[width=1\textwidth]{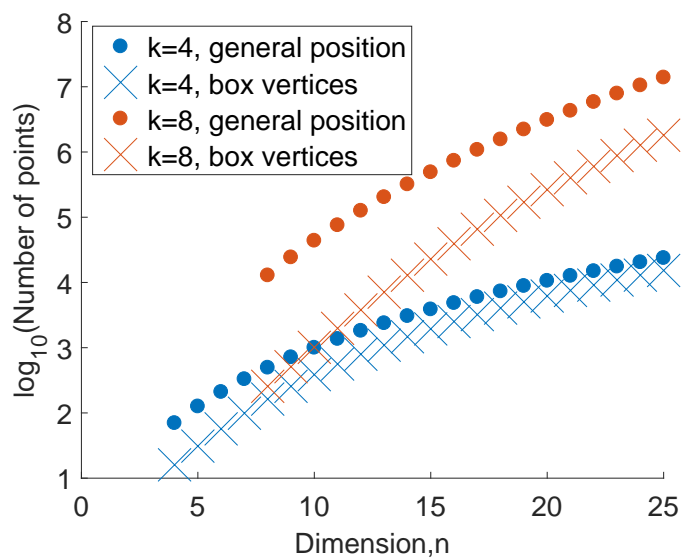}
\caption{The number of points needed to approximate a point in general position (dots), is larger than the number needed to approximate all box vertices (x). Plots show the number of points needed (log-scale) as function of the dimension $n$. Example plots for degrees of approximation $k=4$ (blue) and $k=8$ (orange).}
\label{Theorem_example}
\end{figure}

\section{A function on box can be well approximated using its values on selected vertices}
To simplify notation we work with the standard hypercube $C=[0,1]^n$. Although the results will be valid for general boxes and parallelopipeds.

Let $v$ be a vertex of $C$, we denote by $H(v)$ the Hamming weight of $v$. That is, the number of nonzero coordinates in $v$. We start by proving a lemma which will help in proving theorem \ref{mainThm}.

\begin{lemma}
Let $f$ be a polynomial of degree up to $n-1$, then the following identity is true:
$$\underset{v\in V, H(v)=even}{\sum}f(v)=\underset{v\in V,H(v)=odd}{\sum}f(v)$$
\end{lemma}

\begin{proof}
This linear equation can be checked separately for any monimial of $f$. It is true for any monomial of degree up to $n-1$ for the following reason: the monomial contains at most $n-1$ different variables. Without loss of generality assume that $x_1$ does not appear in the monomial. 
We can separate the terms of the equation in the lemma into pairs $(f(0,x_2,\dots x_n),f(1,x_2,\dots,x_n))$. The elements of each pair are equal and they appear on different sides of the equation of the lemma. Therefore the pairs cancel out and we obtain the equality.  
\end{proof}

As a corollary we obtain that given the values on $2^n-1$ vertices, the value of the remaining vertex can be approximated to $n-1$-th order: $n-1=deg(V-\{v\}\to v)$. This is done using the equation of the lemma, as shown in the following example.

\begin{example}\label{SecOrderOf3d}
Consider the three dimensional case, and let $f$ be a quadratic polynomial, the above lemma explicitly constructs the value $f$ on a vertex given its values on the rest. For example, for the vertex (1,1,1) one obtains:
\begin{multline*}
f(1,1,1)=
f(0,0,0)+f(1,1,0)+f(1,0,1)+f(0,1,1)\\-f(1,0,0)-f(0,1,0)-f(0,0,1)
\end{multline*}
\end{example}

We use the lemma to prove the more general theorem:

\begin{theorem}\label{ApproxChosenPoints}
Let $C=[0,1]^n$ be a hypercube and let $f$ be a polynomial of degree up to $k$. The values of $f$ on the hypercube vertices $v$ with $H(v)\leq k$ determine its values on all hypercube vertices.
\end{theorem}
\begin{proof}
Apply the lemma repeatedly. Use it first to compute $f(v)$ for all vertices with $H(v)=k+1$ to $k$-th order, this can be done since for each vertex $v$ with $H(v)=k+1$ there is a $k+1$ dimensional sub-hypercube for which $v$ is a vertex and the rest of the vertices satisfy $H(v)\leq k$. Then use those values to compute $f(v)$ for $H(v)=k+2$ vertices, etc. until obtaining an approximation for all hypercube vertices.
\end{proof}

Note that theorem \ref{mainThm} follows from the above. Indeed the number of vertices with $H(v)\leq k$ is $\sum_{i=0}^k{n \choose i}$.

\begin{example}
Say we have $12$ different possible mutations and we wish to approximate a fitness function to the second order at all $2^{12}=4096$ mutation combinations, in order to generate an approximate fitness landscape. It is enough to measure the fitness of the wildtype (the case with no mutations), all the single mutations and all pairs of mutations, these are $1+12+{12 \choose 2}=79$ measurements, in order to get this approximation. If we wanted a second order approximation of general points in $12$ dimensions, we must use ${n+k \choose k}={12+2 \choose 2}=91$ points. If we wish to estimate the fitness landscape to third order, we need  $1+12+{12 \choose 2}+{12 \choose 3}=299$ instead of ${n+k \choose k}={12+3 \choose 3}=455$ needed for points in general position.
\end{example}

\begin{remark}\label{NumPointsTight}
The statement of theorem \ref{mainThm} is tight. i.e. there is no approximation of order $k$ to all $n$-dimensional box vertices using less than $\sum_{i=0}^k{n \choose i}$ values at vertices.
\end{remark}

\begin{proof}
Let $M_{ij}=f_i(v_j)$ be the matrix where $f_i$ is the complete set of independent nomomials of degree up to $k$ and $v_j$ the vertices of the hypercube $C$. We need to show that $rank(M)=\sum_{i=0}^k{n \choose i}$. We already know that $rank(M)\leq\sum_{i=0}^k{n \choose i}$ because from this number of columns is enough to obtain all columns of $M$ by linear combinations, as explained in the proof of theorem \ref{ApproxChosenPoints}. We have to check that $rank(M)\geq\sum_{i=0}^k{n \choose i}$. Consider the subset of rows of $M$ defined by all squarefree monomials (e.g. $x$ and $xy$ are in $x^2$ and $x^3$ are out). There are exactly $\sum_{i=0}^k{n \choose i}$ such rows, and we will show that they are independent.

To do so we order the rows first by decreasing Hamming weight, and then by  lexicographic order, for example in the case $k=2,n=3$ we get: $x_1x_2,x_1x_3,x_2x_3,x_1,x_2,x_3,1$. We claim that for each row there is a column which is 0 in all rows above and 1 in this row, this will prove the rows are linearly independent.   

Given a monomial $f$ we associate to it a vertex of the hypercube defined by the variables it includes $v(f)$ (for instance the monomial $f=x_2x_3$ will have the associated vertex $v(f)=(0,1,1)$). Note that the matrix element in row $f$ and column $v(f)$ is 1. Also note that for all rows $f_i$ above $f$, the element in row $f_i$ and column $v(f)$ is zero: indeed, by our ordering, $H(f_i)\geq H(f)$ but they are not equal, hence $f_i$ contains a variable not in $f$.

We conclude that $M$ with this new rows and corresponding columns is a lower traingular square matrix with ones on the diagonal, hence of full rank.

\end{proof}

\begin{example}
For $n=4,k=2$ the original matrix constructed in the proof has ${4+2 \choose 2}=15$ rows and $2^4=16$ columns. The proof above gives a square triangular matrix of size $\sum_{i=0}^2{4 \choose i}=11$ as follows (columns for vertices of hypercube, rows for second order monomials): 

\tiny
\[ 
    \bordermatrix{ & 1100 & 1010 & 1001 & 0110 & 0101 & 0011 & 1000 & 0100 & 0010 & 0001 & 0000  \cr
      x_1x_2 & 1 & 0 & 0 & 0 & 0 & 0 & 0 &0 & 0 & 0 & 0 \cr
      x_1x_3 & 0 & 1 & 0 & 0 & 0 & 0 & 0 &0 & 0 & 0  & 0\cr
      x_1x_4 & 0 & 0 & 1 & 0 & 0 & 0 & 0 &0 & 0 & 0& 0 \cr
      x_2x_3 & 0 & 0 & 0 & 1 & 0 & 0 & 0 &0 & 0 & 0 & 0 \cr
      x_2x_4 & 0 & 0 & 0 & 0 & 1 & 0 & 0 &0 & 0 & 0& 0  \cr
      x_3x_4 & 0 & 0 & 0 & 0 & 0 & 1 & 0 &0 & 0 & 0 & 0 \cr
      x_1 & 1 & 1 & 1 & 0 & 0 & 0 & 1 &0 & 0 & 0 & 0 \cr
      x_2 & 1 & 0 & 0 & 1 & 1 & 0 & 0 &1 & 0 & 0 & 0  \cr
      x_3 & 0 & 1 & 0 & 1 & 0 & 1 & 0 &0 & 1 & 0& 0   \cr
      x_4 & 0 & 0 & 1 & 0 & 1 & 1 & 0 &0 & 0 & 1& 0   \cr
      1 & 1 & 1 & 1 & 1 & 1 & 1 & 1 &1 & 1 & 1& 1} \qquad
\]
\normalsize

\end{example}

\section{An algorithm to check if a set of vertices of a box is enough to approximate any function to a given order}
We know from Theorem \ref{ApproxChosenPoints} that there are sets of $m=\sum_{i=0}^k{n \choose i}$ vertices which allow us to compute all $2^n$ values of a polynomial $f$ of order $k$. Given the value of $f$ on an arbitrary set of vertices, we want to ask to which order one can approximate the values of $f$ at all of the other vertices of the box. Note that the size of the set alone does not determine the order of approximation, as in the example of Figure \ref{ApproxNonapproxExampleFig}.

\begin{figure}
\centering
\includegraphics[width=1\textwidth]{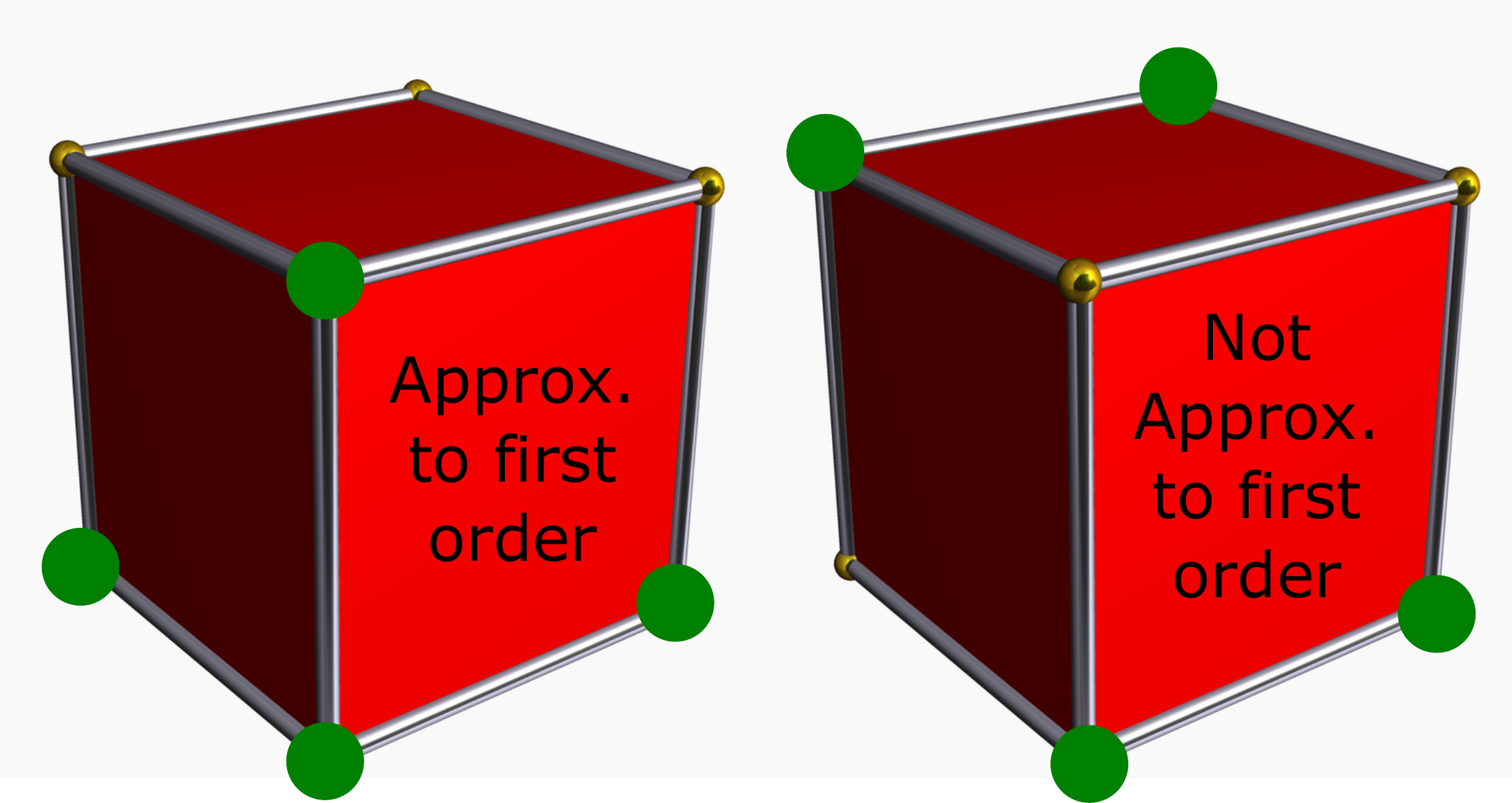}
\caption{An example of two subsets of vertices of a 3d box (green dots). The value of a function on the left set of vertices can be used to approximate it to the first order on any vertex of the box, whereas the values on the right set of vertices cannot.}
\label{ApproxNonapproxExampleFig}
\end{figure}

The idea of the algorithm is as follows. Given a set $V$ of vertices, we obtain the corresponding columns of the matrix defined in Remark \ref{NumPointsTight}, and following the idea of the proof of that remark, we want this submatrix to be of full rank ($\sum_{i=0}^k{n \choose i}$). A full rank gaurantees that any polynomial of degree up to $k$ can be computed on any hypercube vertex using the values at the vertices $V$.   
This linear algebra reformulation provides an efficient algorithms for the following problems:
\begin{enumerate}
\item
Given a set of vertices $V$ of the box and another vertex $v$, to which order we can approximate $f(v)$ knowing only $f(V)$? A specific approximation can also be computed. 
\item
Given a set of vertices $V$, can we approximate all the vertices of the box, to which order? Again, the approximations can be given (each approximation can be computed in polynomial time. Since there are $2^n$ such approximation, all of the approximations together cannot be computed in polynomial time ).  
\end{enumerate}

For example, we provide an algorithm to  compute an approximation for a vertex. The other algorithms can be deduced similarly:  
\begin{algorithm}\label{alg}
Input: A set $V$ of vertices of the hypercube, the values $f(V)$ of the function $f$, another vertex $v$ and a natural number $k$.

Output: An approximation of $f(v)$ to the $k$-th order.

\begin{itemize}
\item
For each vertex in $v_i\in V$ and for $v$ write the corresponding column as in Remark \ref{NumPointsTight}. It gives vectors $u_i,u$ corresponding to $v_i,v$ of length $\sum_{i=0}^k{n \choose i}$ (It is enough to consider only the rows we considered in the proof of theorem \ref{NumPointsTight}).
\item
Write $u$ as a linear combination of the vectors corresponding to $V$: $u=\sum a_iu_i$ (if this is impossible, an approximation does not exist;  return error).
\item
Return $f(v)=\sum a_i f(v_i)$ as the desired approximation 
\end{itemize} 
\end{algorithm}

\section{A random set of vertices linearly approximate the rest with high probability}

In some applications, we obtain values of $f$ on random sets of vertices and we seek an approximation of higher order. An example is fitness landscape evolutionary experiments for which we measure a set of mutations which occur randomly during the evolutionary process (for example \cite{Sarkisyan2016} for random mutations and \cite{Kvitek2011} for evolution). We concentrate here on the case of linear approximation $k=1$. We are looking for the probability that a set of hypercube vertices of cardinality $n+1$ will be affinely independent, which is equivalent to be able to approximate all vertices to the first order.

Currently, the exact probability is not known, but there is an asymptotic upper bound as $n\to \infty$. We are looking for the probability of a random 0-1 matrix to be linearly independent. There is a lower bound for this given by $1-(1/\sqrt{2}+o(1))^n$ \cite{Kahn1995,Tao2007,Bourgain2010}. It is conjectured that the exact asymptotics is given by $1-(1+o(1))n^2/2^n$. Note that this asymptotics reflects the probability that all rows of the matrix are distinct from each other (i.e. not choosing the same vertex of the box twice).

 For smaller values of $n$, although we do not know how to compute the probabilities over the $\RR$, we can compute it over $\mathbb{F}_2$ instead:

\begin{proposition}\label{F2prob}
Consider the hypercube of dimension $n$ over the field $\mathbb{F}_2$. The probability of $n+1$ points to be affinely independent is $$\frac{2^n(2^n-1)(2^n-2)(2^n-4)\dots(2^n-2^{n-1})}{2^n(2^n-1)(2^n-2)(2^n-3)\dots(2^n-n)}$$ The probability monotonically decreases and converges when $n\to\infty$ to a finite value  $(\frac{1}{2};\frac{1}{2})_{\infty}\approx 0.288$, where $(\frac{1}{2};\frac{1}{2})_{\infty}$ denotes the q-Pochhammer symbol with $q=1/2$ \cite{Andrews1986}.
\end{proposition} 

\begin{proof}
The number of possible choices of subsets of vertices of the hypercube of cardinality $n+1$ is ${2^n \choose n+1}$. To choose an affinely indepedent set  we have $\frac{2^n(2^n-1)(2^n-2)(2^n-4)\dots(2^n-2^{n-1})}{n+1!}$ options, this expression was computed as the number of options to choose the new vertex affinely independent on the previuos ones, divided by all possible orders. Hence the probability for independent set over $\mathbb{F}_2$ is  $\frac{2^n(2^n-1)(2^n-2)(2^n-4)\dots(2^n-2^{n-1})}{2^n(2^n-1)(2^n-2)(2^n-3)\dots(2^n-n)}$, we divide the numerator and denominator by $2^{n(n+1)}$ and obtain that for large $n$ the denominator ${\displaystyle \lim_{n \to \infty} \prod_{m=0}^{n} (1-m/2^n)}=1$ and the numerator ${\displaystyle \lim_{n \to \infty} \prod_{m=0}^{n-1} (1-2^{m-n})}=(\frac{1}{2};\frac{1}{2})_{\infty}\approx 0.288$ is the q-Pochhammer symbol. It remains to show that the sequence is monotonically decreasing, to do so we compute the ratio:

\begin{multline*}
\frac{a_{n+1}}{a_n}=\frac{2^{(n+1)}(2^{(n+1)}-1)(2^{(n+1)}-2)(2^{(n+1)}-4)\dots(2^{(n+1)}-2^n)}{2^{(n+1)}(2^{(n+1)}-1)(2^{(n+1)}-2)(2^{(n+1)}-3)\dots(2^{(n+1)}-(n+1))}\cdot \\ \cdot \frac{2^n(2^n-1)(2^n-2)(2^n-3)\dots(2^n-n)}{2^n(2^n-1)(2^n-2)(2^n-4)\dots(2^n-2^{n-1})}=\\
=\frac{1(1-1/2^{(n+1)})(1-2/2^{(n+1)})(1-4/2^{(n+1)})\dots(1-1/2)}{1(1-1/2^{(n+1)})(1-2/2^{(n+1)})(1-3/2^{(n+1)})\dots(1-(n+1)/2^{(n+1)})}\cdot \\ \cdot \frac{1(1-1/2^n)(1-2/2^n)(1-3/2^n)\dots(1-n/2^n)}{1(1-1/2^n)(1-2/2^n)(1-4/2^n)\dots(1-1/2)}=\\
= \frac{(1-1/2^n)(1-2/2^n)(1-3/2^n)\dots(1-n/2^n)}{(1-2/2^{(n+1)})(1-3/2^{(n+1)})\dots(1-(n+1)/2^{(n+1)})}<1
\end{multline*}
Where the last inequality follows by elementwise comparison of the numerator and denominator, and true for $n>1$ (for $n=1$ there is an equality).
\end{proof}

Note that the probability computed above for $\mathbb{F}_2$ is a lower bound on the probability seek, indeed:

\begin{proposition}
If a set of vertices is affinely indepedent over $\mathbb{F}_2$, it is also affinely independent over $\RR$.
\end{proposition}

\begin{proof}
Without loss of generality assume that the origin is in the set of vertices, otherwise apply a symmetry on the hypercube such that this is the case. We need to show that the rest of vertices are linearly independent. We show conversely, that if the set is linearly dependent over $\RR$ it is also linearly dependent over  $\mathbb{F}_2$. Indeed, by assumption there is a linear combination $\sum a_iv_i=0$ With $a_i\in \RR$, this $a_i$ can be chosen rational, since all vertices of the hypercube have rational coefficient. If $a_i$ are not integral, we multiply by the common denominator of the $a_i$ to make them so. If all new $a_i$ are even, we divide by the maximal power of two dividing all of them, we now obtained $a_i$ which are integral, not all even and $\sum a_iv_i=0$.  We now take this equation mod 2 and see that the vertices are depedent over $\mathbb{F}_2$.
\end{proof}

 Using algorithm \ref{alg} we can compute the real probabilities of approximation for small values of $n$, we plot this probabilities for the first order approximation and the $\mathbb{F}_2$ lower bound in Figure \ref{LinearApproxFig}, for very small $n$ the approximation is fine, but for larger $n$ the real probability is increasing, while the $\mathbb{F}_2$ bound is decreasing to $0.288$. The increasing probabilites mean that a random set of $n+1$ mutations is with high probability useful in approximating the entire fitness landscape to the first order.

\begin{figure}\label{LinearApproxFig}
\centering
\includegraphics[width=1\textwidth]{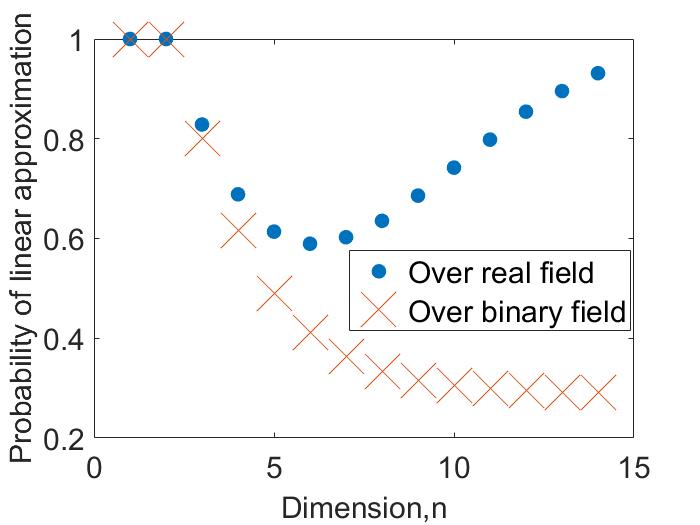}
\caption{The probability of random $n+1$ values of $f$ on approximate all the box vertices as function of the dimension. The probability over $\mathbb{F}_2$ was computed using proposition \ref{F2prob}. For the real case, the values of 1-6 were computed presicely by inspecting all possible subsets of the box and using algorithm \ref{alg}. The values 7-14 were approximated using a Monte-Carlo simulation, 100000 random subsets were selected and we counted how many of them approximate the entire box using algorithm \ref{alg}.}
\label{LinearApproxFig}
\end{figure}

\section{Conclusions}
We showed that the biological applications of predicting the effect of drug combinations and estimating values in fitness lndscape can be modelled as approximation problems of functions on box vertices. We defined it formally using algebraic geometry and the zero locus of polynomials of given degrees, and proved that with the correct choice of box vertices, these problems can be solved better than expcted in terms of degree of approximation for a given number of vertices used. Specifically, for a box of dimension $n$ and a desired approximation degree $k$, $\sum_{i=0}^k{n \choose i}$ given vertices are suffice for approximation of all vertices, instead of $\binom{n+k}{k}$ expected if points were in general position.
 We also discussed the case where we do not choose the points, in the case of linear approximation and for large values of  $n$, the probability to obtain linear approximation using $n+1$ points exponentially close to 1.

\bibliographystyle{plain}
\bibliography{approx_hypercube_arxiv3.bbl}

\begin{thebibliography}{10}

\bibitem{Andrews1986}
George~E. Andrews and {American Mathematical Society.}
\newblock {\em {Q-series : their development and application in analysis,
  number theory, combinatorics, physics, and computer algebra}}.
\newblock Published for the Conference Board of the Mathematical Sciences by
  the American Mathematical Society, 1986.

\bibitem{Bourgain2010}
Jean Bourgain, Van~H. Vu, and Philip~Matchett Wood.
\newblock {On the singularity probability of discrete random matrices}.
\newblock {\em Journal of Functional Analysis}, 258(2):559--603, jan 2010.

\bibitem{DeVita1975}
V~T DeVita, R~C Young, and G~P Canellos.
\newblock {Combination versus single agent chemotherapy: a review of the basis
  for selection of drug treatment of cancer.}
\newblock {\em Cancer}, 35(1):98--110, jan 1975.

\bibitem{Eisenbud1996}
David Eisenbud, Mark Green, and Joe Harris.
\newblock {CAYLEY-BACHARACH THEOREMS AND CONJECTURES}.
\newblock {\em BULLETIN (New Series) OF THE AMERICAN MATHEMATICAL SOCIETY},
  33(3), 1996.

\bibitem{Jin2005}
Y.~Jin.
\newblock {A comprehensive survey of fitness approximation in evolutionary
  computation}.
\newblock {\em Soft Computing}, 9(1):3--12, jan 2005.

\bibitem{Kahn1995}
Jeff Kahn and Janos Komlos.

\bibitem{Kvitek2011}
Daniel~J. Kvitek and Gavin Sherlock.
\newblock {Reciprocal Sign Epistasis between Frequently Experimentally Evolved
  Adaptive Mutations Causes a Rugged Fitness Landscape}.
\newblock {\em PLoS Genetics}, 7(4):e1002056, apr 2011.

\bibitem{Morgan2011}
Steve Morgan, Paul Grootendorst, Joel Lexchin, Colleen Cunningham, and Devon
  Greyson.
\newblock {The cost of drug development: A systematic review}.
\newblock {\em Health Policy}, 100(1):4--17, apr 2011.

\bibitem{Sarkisyan2016}
Karen~S. Sarkisyan, Dmitry~A. Bolotin, Margarita~V. Meer, Dinara~R. Usmanova,
  Alexander~S. Mishin, George~V. Sharonov, Dmitry~N. Ivankov, Nina~G.
  Bozhanova, Mikhail~S. Baranov, Onuralp Soylemez, Natalya~S. Bogatyreva,
  Peter~K. Vlasov, Evgeny~S. Egorov, Maria~D. Logacheva, Alexey~S. Kondrashov,
  Dmitry~M. Chudakov, Ekaterina~V. Putintseva, Ilgar~Z. Mamedov, Dan~S. Tawfik,
  Konstantin~A. Lukyanov, and Fyodor~A. Kondrashov.
\newblock {Local fitness landscape of the green fluorescent protein}.
\newblock {\em Nature}, 533(7603):397--401, may 2016.

\bibitem{Tao2007}
Terence Tao and Van Vu.
\newblock {On the singularity probability of random Bernoulli matrices}.
\newblock {\em Journal of the American Mathematical Society}, 20(03):603--629,
  jul 2007.

\bibitem{Wood2012}
Kevin Wood, Satoshi Nishida, Eduardo~D Sontag, and Philippe Cluzel.
\newblock {Mechanism-independent method for predicting response to multidrug
  combinations in bacteria.}
\newblock {\em Proceedings of the National Academy of Sciences of the United
  States of America}, 109(30):12254--9, jul 2012.

\end{thebibliography}

\end{document}